\documentclass[12pt]{article}

\usepackage[hmargin=1in,vmargin=1in]{geometry}

\usepackage{amssymb,amsmath,color}
\usepackage[dvipsnames]{xcolor}
\usepackage{amsthm}
\usepackage{url}
\usepackage{tikz}
\usepackage{hyperref}

\definecolor{red}{rgb}{1,0,0}
\definecolor{blue}{rgb}{.2,.2,.8}

\newtheorem{theorem}{Theorem}
\newtheorem{corollary}[theorem]{Corollary}
\newtheorem{proposition}[theorem]{Proposition}
\newtheorem{conjecture}[theorem]{Conjecture}
\newtheorem{lemma}[theorem]{Lemma}
\newtheorem{question}[theorem]{Question}

\theoremstyle{definition}

\newtheorem{example}[theorem]{Example}

\def\la{\lambda}
\def\cP{\mathcal P}
\def\cQ{\mathcal Q}
\def\cB{\mathcal B}
\def\cS{\mathcal S}

\def\ra{\rightarrow}
\def\bth{\begin{theorem}}
\def\eth{\end{theorem}}
\def\ble{\begin{lemma}}
\def\ele{\end{lemma}}
\def\bco{\begin{corollary}}
\def\eco{\end{corollary}}
\def\bpr{\begin{proposition}}
\def\epr{\end{proposition}}
\def\bcon{\begin{conjecture}}
\def\econ{\end{conjecture}}
\def\bqu{\begin{question}}
\def\equ{\end{question}}
\def\bprf{\begin{proof}}
\def\eprf{\end{proof}}
\def\ben{\begin{enumerate}}
\def\een{\end{enumerate}}
\def\beq{\begin{equation}}
\def\eeq{\end{equation}}
\def\De{\Delta}
\def\de{\delta}
\def\be{\beta}

\def\cBR{\mathcal{B}\mathcal{R}}
\def\cPR{\mathcal{P}\mathcal{R}}

\newcommand{\fl}[1]{\lfloor #1 \rfloor}
\newcommand{\ce}[1]{\lceil #1 \rceil}
\newcommand{\flf}[2]{\left\lfloor\frac{#1}{#2}\right\rfloor}

\DeclareMathOperator{\pre}{pre}
\DeclareMathOperator{\prh}{prh}
\DeclareMathOperator{\ImP}{ImP}
\DeclareMathOperator{\ImB}{ImB}
\DeclareMathOperator{\Mod}{mod}

\begin{document}

\title{Elementary symmetric partitions}

\author{
Cristina Ballantine\\[-5pt]
\small Department of Mathematics and Computer Science, College of the Holy Cross,\\[-5pt]
\small Worcester, MA 01610, USA, {\tt cballant@holycross.edu}
\\
George Beck\\[-5pt]
\small Department of Mathematics and Statistics, Dalhousie University,\\[-5pt]
\small Halifax, NS, B3H 4R2, Canada, {\tt george.beck@gmail.com}
\\
Mircea Merca\\[-5pt]
\small Department of Mathematical Methods and Models,\\[-5pt] 
\small  Fundamental Sciences Applied in Engineering Research Center\\[-5pt]
\small National University of Science and Technology Politehnica Bucharest,\\[-5pt]
\small RO-060042 Bucharest, Romania, {\tt mircea.merca@upb.ro}\\[-5pt]
\small and\\[-5pt]
\small Academy of Romanian Scientists, \\[-5pt]
\small RO-050044 Bucharest, Romania
\\
Bruce E. Sagan\\[-5pt]
\small Department of Mathematics, Michigan State University,\\[-5pt]
\small East Lansing, MI 48824, USA, {\tt sagan@math.msu.edu}
}

\date{\today\\[10pt]
	\begin{flushleft}
	                                       \\[5pt]
	\end{flushleft}}

\maketitle

\begin{abstract}
Let $e_k(x_1,\ldots,x_\ell)$ be an elementary symmetric polynomial and let $\la=(\la_1,\ldots,\la_\ell)$ be an integer partition.  Define $\pre_k(\la)$ to be the partition whose parts are the summands in the evaluation $e_k(\la_1,\ldots,\la_\ell)$.  The study of such partitions was initiated by Ballantine, Beck, and Merca who showed (among other things) that $\pre_2$ is injective as a map on binary partitions of $n$.  In the present work we derive a host of identities involving the sequences which count the number of parts of a given value in the image of $\pre_2$.  These include generating functions, explicit expressions, and formulas for forward differences.  We generalize some of these to $d$-ary partitions and explore connections with color partitions.  Our techniques include the use of generating functions and bijections on rooted partitions.  We end with a list of conjectures and a direction for future research.
\end{abstract}


\section{Definitions and introduction}

The focus of this paper will be integer partitions.
For any set $S$ we denote the cardinality of $S$ by $|S|$.
Let $\la=(\la_1,\la_2,\ldots,\la_\ell)$ be an {\em integer partition of $n$}, that is, a weakly decreasing sequence of positive integers whose sum is $n$.  The $\la_i$ are called {\em parts}.  We let
$$
|\la| =\sum_i \la_i
$$
and 
$$
\ell(\la) = \ell.
$$
We call $|\la|$  the {\em weight} or {\em size} of $\la$, and $\ell(\la)$ is the {\em length}.  The use of vertical bars for cardinality or weight should be distinguished by context.  
Note that if $n<0$ then no such partitions exists, while there is a unique partition of $0$, namely the empty partition $\epsilon$.
Also define
$$
\cP(n)=\{\la \mid  \text{$\la$ is a partition of $n$}\},
$$
as well as
$$
p(n) =|\cP(n)|,
$$
and
$$
\cP_k(n) = \{\la\in\cP(n) \mid \ell(\la)\ge k\}.
$$
For example,
$$
\cP_3(5) = \{(1,1,1,1,1),\ (2,1,1,1),\ (2,2,1),\ (3,1,1)\}.
$$

We will be particularly interested in {\em binary partitions}, which are those such that all parts are a power of $2$. 
Analogous to the notation just defined, let
$$
\cB(n)=\{\la \mid  \text{$\la$ is a binary partition of $n$}\},
$$
and
$$
\cB_k(n) = \{\la\in\cB(n) \mid \ell(\la)\ge k\}.
$$
To illustrate,
$$
\cB_3(5) = \{(1,1,1,1,1),\ (2,1,1,1),\ (2,2,1)\}.
$$

We will be combining partitions and elementary symmetric polynomials.  The {\em $k$th elementary symmetric polynomial} in a set of variables 
$\{x_1,x_2,\ldots,x_\ell\}$  is
$$
e_k(x_1,x_2,\ldots,x_\ell) = \text{ the sum of all square-free, degree $k$ monomials in the $x_i$}.
$$ 
Note that if $\ell<k$ then there are no such monomials so that $e_k(x_1,x_2,\ldots,x_\ell)$ is the empty sum and hence equals $0$.  For a less trivial example,
$$
e_2(x_1,x_2,x_3,x_4) = x_1 x_2 + x_1 x_3 + x_1 x_4 + x_2 x_3 + x_2 x_4 + x_3 x_4.
$$
Given a partition $\la=(\la_1,\la_2,\ldots,\la_\ell)$ with $\ell\ge k$, we define $\pre_k(\la)$ to be the partition whose parts are the summands in the evaluation $e_k(\la_1,\la_2,\ldots,\la_\ell)$.

To illustrate, if $\la=(3,2,1,1)$ then
$$
e_2(3,2,1,1) = 3\cdot 2 + 3\cdot 1 + 3\cdot 1 + 2\cdot 1 + 2\cdot 1 + 1\cdot 1
$$
so that
$$
\pre_2(3,2,1,1) = (6,3,3,2,2,1).
$$
We call $\pre_k(\la)$ an {\em elementary symmetric partition}.
Thus we have defined a map 
$$
\pre_k:\cP_k(n)\ra\cP,
$$
where $\cP$ is the set of all integer partitions.

The function $\pre_k$ was first studied by Ballantine, Beck and Merca~\cite{bbm:pesp}.  In particular, they made the following conjecture, among others.
\begin{conjecture}[\cite{bbm:pesp}]
\label{prek:con}
For any $k\ge1$ and $n\ge0$, the map $\pre_k:\cP_k(n)\ra\cP$ is injective.   
\end{conjecture}
Note that $\pre_1$ is clearly injective since $\pre_1(\la)=\la$ for any partition.
The following weaker form of the conjecture was proved in [\cite{bbm:pesp}], where one restricts $\pre_2$ to binary partitions. In it, $\cB$ is the set of all binary partitions. 
\bth[\cite{bbm:pesp}]
For any $n\ge0$, the map $\pre_2:\cB_2(n)\ra\cB$ is injective.\hfill\qed
\eth
Because of this result, we will focus part of this work on binary partitions.
We will use the notation
$$
\ImP_k(n) = \pre_k(\cP_k(n))
$$
and
$$
\ImB_k(n) = \pre_k(\cB_k(n))
$$
for the images of the $\pre_k$ function applied to the full domain of partitions or just the binary partitions, respectively, of length at least $k$.

We will be interested in counting multiplicities in the image sets of $\pre_2$.  If $\la$ is a partition and  $i$ is a positive integer  then the {\em multiplicity} of $i$ in $\la$, denoted $m_i(\la)$, is the number of times $i$ occurs as a part in $\la$.  Extend this notation to finite sets $S$ by
$$
m_i(S) = \sum_{\la\in S} m_i(\la).
$$
  
Our techniques will include manipulation of generating functions and bijections.  For the latter, we modify the notion of a partition as follows.
A \textit{rooted partition} is a partition $\la$ with one or more parts  distinguished and  
marked with hats.  
These special parts are called the {\em roots} of $\la$.
The position   of the roots  among the other parts of the same size matters. 
For example, 
$$
\la = (4,3,\hat{3},3,3,\hat{1},1,\hat{1})
$$
has three roots, namely $\hat{3}$ and two copies of $\hat{1}$.  Furthermore, it is different from the rooted partition
$$
\la' = (4,3,3,\hat{3},3,\hat{1},1,\hat{1}).
$$
Rooted partitions were introduced by Sagan~\cite{sag:rpn} in order to give combinatorial proofs of some results of Merca and Schmidt~\cite{MS:pir,MS:pfp} involving the partition, Euler totient, and M\"obius functions.  We denote by 
\begin{align*}
\cBR_i(n) &= \text{ the set of rooted binary partitions of $n$ with a single root $\hat{\imath}$},\\
\cBR_{i,j}(n) &= \text{ the set of rooted binary partitions of $n$ with exactly two roots $\hat{\imath}$ and 
$\hat{\jmath}$}.
\end{align*}
The reason that rooted partitions will be useful is as follows.  Suppose that a part $k$ in 
$\pre_2(\la)$ came from multiplying parts $i$ and $j$ of $\la$.  Then one can think of $k$ as being associated with the partition obtained by rooting $\la$ at $i$ and $j$.  Thus we can transfer information about the parts of the range partition into properties of doubly rooted partitions in the domain which will be easier to work with. 
This will be made precise in Lemma~\ref{2^k:ImB} below.

There are two operations which will be useful for us when working with (rooted) partitions.
If $\la$ and $\mu$ are partitions then their {\em direct sum} is  $\la\oplus\nu$ which is obtained by, for each $i$, concatenating the string of $i$'s in $\la$ with the string of $i$'s in $\nu$, including any $\hat{\imath}$'s which exist.  It is important when considering roots to have the parts from $\la$ before those in $\mu$.  For example
$$
(3,\hat{3},3,2,1,1)\oplus (3,2,2,\hat{1},1) = (3,\hat{3},3,3,2,2,2,1,1,\hat{1},1).
$$
For the second operation, if  $m$ is a positive integer and $\la$ is a partition of $n$  then we write $m\la$ for the partition of $mn$ whose parts are the parts of $\la$ multiplied by $m$. If $\mu$ is a partition of $n$ such that all  parts of $\mu$ are multiples of $m$, we write $\mu/m$  for the partition of $n/m$ whose parts are the parts of $\mu$ divided by $m$. 

The rest of this article is structured as follows.  In the next section we will investigate the sequences $m_i(\ImB_2(n))$ for $i=1$, $2$, and $4$.  We will derive generating functions as well as expressions for both the sequences themselves and their forward differences.  Section~\ref{dp} will extend various  results on binary partitions to those which are $d$-ary in that every part is a power of a fixed $d$.  In particular, we will show that $\pre_2$ is injective when restricted to $d$-ary partitions of $n$ and consider the number of $1$'s and $d$'s in the image.
Section~\ref{cp} shows that there is a close relationship between the multiplicity sequences we are studying and color partitions.  A color partition is one where the parts have been assigned certain colors such that the number of possible colors depends on the size of the part.  We end with a section of conjectures and a direction for future work.

\section{Binary partitions}

It will be useful for our combinatorial proofs to have an interpretation of the number of times $2^k$ appears in $\ImB_2(n)$ in terms of rooted partitions. 
\ble
\label{2^k:ImB}
For $k\ge0$
$$
m_{2^k}(\ImB_2(n)) = 
\sum_{i=0}^{\fl{k/2}}\, |\cBR_{2^i,2^{k-i}}(n)|.
$$
\ele
\bprf
Since the map $\pre_2$ is injective on binary partitions, each part equal to $2^k$ in $\mu\in\ImB_2(n)$ comes from multiplying two parts $2^i$ and $2^{k-i}$ in its preimage
$\la\in\cB(n)$.  And to prevent double counting we must assume $i\le\fl{k/2}$.
But such a pair can be associated with a rooting of $\la$ at these two elements, and the lemma follows.
\eprf

The next theorem concerns the number $a_n$ of $1$'s in $\ImB_2(n)$.  It gives a generating function, various formulas for $a_n$ including  connections with a sequence from the OEIS  and with the sum-of-divisors function, as well as  expressions for the corresponding difference sequence.
For any OEIS sequence, we will use $S(n)$ for the $n$th term of $S$, where $S$ is the sequence number in the Encyclopedia.  
Also, we use the {\em forward difference operator} on any sequence $(s_n)_{n\ge1}$, which is the sequence 
 with $n$th element
 $$
 \De s_n = s_{n+1}-s_n.
 $$ 
 Finally, recall the classical formula
$$
n\,p(n) = \sum_{k=1}^n \sigma_1(k)\, p(n-k),
$$
where $\sigma_1(k)$ is the sum of the positive divisors of $k$.  For our analogue of this identity we let
\beq
\label{be(n)}
\be(n) = \sum_{2^i | n} 2^i
\eeq
be the sum of the divisors of $n$ which are powers of two.
We note that part (g) of the following theorem was proved using algebraic techniques in~\cite[Theorem 11]{bbm:pesp}. Here we give a combinatorial proof using doubly rooted partitions. 
\bth
\label{m1B}
For $n\ge0$, let
$$
a_n = m_1(\ImB_2(n)).
$$
\ben
\item[(a)] We have
$$
\sum_{n\ge0} a_n\, q^n = \frac{q^2}{(1-q)^2}\ \prod_{i\ge0} \frac{1}{{1-q^{2^i}}}.
$$
\item[(b)]  For $n\ge0$, 
$$
a_n = \sum_{i\ge1} (i-1)\ |\cB(n-i)|.
$$
\item[(c)] For $n\ge0$,
$$
a_n = \sum_{i\ge 0} \binom{n-2i}{2}\ |\cB(i)|.
$$
\item[(d)] For $n\ge0$,
$$
(n-2)\,a_n=\sum_{k=1}^n (\be(k)+ 2)\,a_{n-k}.
$$
\item[(e)]  The sequence $(a_n)_{n\ge1}$ satisfies $a_1=0$ and, for $n\ge 2$,
$$
a_n = {\rm A}131205(n-1).
$$
\item[(f)]  For $n\ge0$, there is this connection with the sequence \emph{A000123} via
$$
\De a_n = |\cB(2n-2)|.
$$

\item[(g)] For $n\ge0$,
$$
\De a_n =  a_{\fl{(n+2)/2}} + a_{\ce{(n+2)/2}}.
$$
\een
\eth
\bprf
(a)  From the case $k=0$ of Lemma~\ref{2^k:ImB} we have
\beq
\label{a=BR}
a_n=|\cBR_{1,1}(n)|.
\eeq
Say that the roots of partition $\la$ are at positions $j$ and $j+k$ where $j,k\ge1$.  Furthermore, the remaining parts of $\la$ form a binary partition. 
Thus
$$
\sum_{n\geq 0}a_n\,q^n
=\sum_{j,k\geq 1}q^j\, q^k\ \prod_{i\geq 0}\frac{1}{1-q^{2^i}}
= \frac{q^2}{(1-q)^2}\ \prod_{i\ge0} \frac{1}{{1-q^{2^i}}}.
$$

\medskip

(b) We give two proofs, one algebraic and one combinatorial.  For the first, write the generating function from (a) as
$$
\sum_{n\ge0} a_n\, q^n  = \frac{q^2}{(1-q)^2} \cdot \prod_{i\ge0} \frac{1}{{1-q^{2^i}}}= 
\left(\sum_{n\ge1} (n-1) q^n\right)\cdot \left(\sum_{n\ge0} |\cB(n)| q^n\right)
$$
and take the coefficient of $q^n$ in the first and third expressions.

For the combinatorial proof, we use equation~\eqref{a=BR}. If $\la\in\cBR_{1,1}(n)$ then write
$\la = \mu\oplus  \nu$
where $\nu$ consists of all the elements after and including the first $\hat{1}$. 
Suppose $|\nu|=i$.  Then there are $i-1$ ways to choose the second root in $\nu$.  And there are $|\cB(n-i)|$ ways to choose $\mu$.  Summing gives the result.

\medskip

(c) We give two proofs.  Rewriting (a)  we have 
$$
\sum_{n\geq 0} a_n\,q^n = \frac{q^2}{(1-q)^3} \cdot \prod_{i\geq 1} \frac{1}{1-q^{2^i}} 
= \left( \sum_{n\geq 0} \binom{n}{2}q^n \right) \left( \sum_{n\geq 0} |\cB(n)|\,q^{2n} \right)
$$
and coefficient extraction finishes the demonstration.

Combinatorially, we once more appeal to~\eqref{a=BR} and write 
$\la\in\cBR_{1,1}(n)$ 
as $\la = \mu\oplus  \nu$
where now $\nu$ contains all the $1$'s and 
$\hat{1}$'s.  Since $\mu$ contains all the positive powers of two we have 
$\mu/2\in\cB(i)$ for some $i$ where $|\mu|=2i$.  So $|\nu|=n-2i$ and hence there are $\binom{n-2i}{2}$ ways to choose two $\hat 1$'s.  Finally one sums over all possible $i$.

\medskip

(d) From part (a) we have
\begin{align}
\frac{d}{d q} \ln \left(\sum_{n\ge0} a_n\,q^n \right) 
& = \frac{d}{d q} \ln \left( \frac{q^2}{(1-q)^2} \prod_{i\geq 0} \frac{1}{1-q^{2^i}} \right) \notag \\
& = \frac{d}{d q} \left(\ln\frac{q^2}{(1-q)^2} +  \sum_{i\geq 0} \ln\left( \frac{1}{1-q^{2^i}} \right)\right) \notag \\
& = \frac{2}{q(1-q)} + \sum_{i\geq 0} \frac{2^i\,q^{2^i}}{q(1-q^{2^i})} \notag \\
& = \frac{1}{q}\sum_{n\geq 0} 2\,q^n +\frac{1}{q} \sum_{n\geq 1}  \left(\sum_{2^j|n} 2^j\right)\, q^n \notag \\
& = \frac{1}{q} \left(2+\sum_{n\geq 1} \left(2+\be(n)\right)\, q^n\right).\label{eqdiv}
  \end{align}
From this we deduce
\begin{align*}
    \sum_{n\geq 0} n\,a_n\,q^n &= \left( \sum_{n\geq 0} a_n\,q^n\right)  \left(2+\sum_{n\geq 1} \left(2+\be(n)\right)\, q^n\right).
\end{align*}
Taking the coefficient of $q^n$ finishes the proof.

\medskip

(e)  The generating function in (a) is the same as the one given for A131205 multiplied by $q$.

\medskip

(f)  Again, we give two proofs.  For the algebraic one, let
$$
B(q) = \sum_{n\ge0} |\cB(n)|\ q^n = \prod_{i\ge0} \frac{1}{{1-q^{2^i}}}
$$
and, using part (a), define $F(q)$ to be
$$
F(q) = \sum_{n\ge0} a_n\, q^n = \frac{q^2}{(1-q)^2} B(q).
$$
Then
$$
\sum_{n\ge0} \De a_n\, q^n = \frac{F(q)}{q} - F(q) = \frac{q}{1-q} B(q)
$$
and
\begin{align*}
\sum_{n\ge 0} |\cB(2n)|\, q^{2n}
&= \frac{1}{2}\left( B(q)  + B(-q)\right)\\
&= \frac{1}{2}\left(\frac{1}{1-q} \prod_{i\ge1} \frac{1}{{1-q^{2^i}}}
+\frac{1}{1+q} \prod_{i\ge1} \frac{1}{{1-q^{2^i}}}\right)\\
&= \frac{1}{1-q^2}\prod_{i\ge1} \frac{1}{{1-q^{2^i}}}.
\end{align*}
Substituting $q^{1/2}$ into the generating function just given and comparing it to the one for $\De a_n$ completes the proof.

We also give a combinatorial proof. 
Each $\mu\in\cB(n+1)$ having a pair of ones counted by $a_{n+1}$
can be obtained uniquely from some $\la\in\cB(n)$ as $\mu=\la\oplus (1)$. 
So $\De a_n=a_{n+1}-a_n$ is just the number of new pairs of $1$'s obtained by adding the final $1$. 
The new $1$ creates $m_1(\la)$ new pairs.  So 
\beq
\label{Da=BR}
\De a_n = m_1(\cB(n))=|\cBR_1(n)|,
\eeq
since counting a $1$ is equivalent to rooting it.  Thus it suffices to show that there is a bijection $f:\cBR_1(n)\ra\cB(2n-2)$.  An example illustrating the construction of $f$ follows the proof of this theorem.

Given $\la\in\cBR_1(n)$ we write
$$
\la = \mu \oplus \nu
$$
where $\nu$ contains the $\hat{1}$ and all the $1$'s after it.  If $|\nu|=j$ then we let
$$
f(\la) = 2\mu \oplus (1^{2j-2}).
$$
It is easy to check that $f$ is well-defined in that $f(\la)\in\cB(2n-2)$.  To show that $f$ is bijective, we construct its inverse.  Given $\la'\in\cB(2n-2)$ we write
$$
\la' = \mu' \oplus \nu'
$$
where $\nu'$ contains all the $1$'s in $\la'$.  Now let
$$
f^{-1}(\la') = \mu'/2 \oplus (\hat{1},1^{k/2})
$$
where $k=|\nu'|$.  Again, the verification that $f^{-1}$ is well defined and indeed the inverse to $f$ is straightforward and so left to the reader.

\medskip

(g)  It is not hard to prove this algebraically using part (a), but we will give a combinatorial demonstration.  By equations~\eqref{a=BR} and~\eqref{Da=BR}, it suffices to find a function
$f:\cBR_1(n)\rightarrow R$ where
$$
R = \begin{cases}
    \cBR_{1,1}\left(\frac{n+1}{2}\right)\uplus \cBR_{1,1}\left(\frac{n+3}{2}\right) 
    & \text{if $n$ is odd},\\[5pt]
    \cBR_{1,1}\left(\frac{n+2}{2}\right)
    & \text{if $n$ is even},
\end{cases}
$$
where $f$ is bijective if $n$ is odd, $2$-to-$1$ if $n$ is even, and $\uplus$ is disjoint union.
To define $f$, suppose $\la\in\cBR_1(n)$ and write
$$
\la = \mu\oplus\nu
$$
where $\nu$ contains all the $1$'s and the $\hat{1}$ of $\la$.  Let $i$ and $j$ be the number of $1$'s before and after the $\hat{1}$, respectively.  Let
$$
f(\la) = \mu/2 \oplus (\hat{1}) \oplus (1^{\fl{i/2}}) \oplus (\hat{1}) \oplus (1^{\fl{j/2}}).
$$
Again, $f$ being well-defined is a straightforward, if case-heavy, calculation based on the parities of $i$ and $j$.

To show that $f$ is bijective when $n$ is odd, we construct its inverse.  Write
\beq
\label{la'}
\la' =\mu' \oplus (\hat{1}) \oplus (1^k) \oplus (\hat{1}) \oplus (1^l).
\eeq
for some nonnegative integers $k,l$.  
If 
$\la'\in\cBR_{1,1}\left(\frac{n+1}{2}\right)$ 
then let
$$
f^{-1}(\la') = 2\mu' \oplus (1^{2k+1}) \oplus (\hat{1}) \oplus (1^{2l+1}).
$$
On the other hand, if 
$\la'\in\cBR_{1,1}\left(\frac{n+3}{2}\right)$ 
then let
$$
f^{-1}(\la') = 2\mu' \oplus (1^{2k}) \oplus (\hat{1}) \oplus (1^{2l}).
$$
Again, we omit the tedious details that this is well defined and an inverse.

Finally, we must show that $f$ is $2$-to-$1$ when $n$ is even.  We keep the notation of the previous paragraph for $\la'$ and, in particular, the decomposition~\eqref{la'}.  Then 
$$
f^{-1}(\la') = 
\{
2\mu'\oplus (1^{2k}) \oplus (\hat{1}) \oplus (1^{2l+1}),\quad
2\mu'\oplus (1^{2k+1}) \oplus (\hat{1}) \oplus (1^{2l})
\}.
$$
As usual, we leave the verification to the reader.
\eprf

We now illustrate the bijection $f$ in part (f) of the previous proof.  

\begin{example}Suppose that
$$
\la = (4,2,2,1,\hat{1},1,1) = (4,2,2,1)\oplus(\hat{1},1,1)
$$
so that
$$
\mu = (4,2,2,1) \text{ and } \nu = (\hat{1},1,1).
$$
Now $j=|\nu|=3$ so that $2j-2 = 4$ so that 
$$
f(\la) = 2(4,2,2,1)\oplus(1^4) = (8,4,4,2,1,1,1,1).
$$
For the inverse, consider
$$
\la'= (8,4,4,2,1,1,1,1) = (8,4,4,2) \oplus (1,1,1,1)
$$
which implies
$$
\mu' = (8,4,4,2) \text{ and } \nu'=(1,1,1,1).
$$
This means that $k=|\nu'|=4$ and $k/2=2$, so
$$
f^{-1}(\la') = (8,4,4,2)/2 \oplus (\hat{1},1^2) = (4,2,2,1,\hat{1},1,1)
$$
recovering the original partition $\la$.
\end{example}

We now consider $2$'s in $\ImB_2(n)$. 
\bth \label{m2B}
For $n\ge0$, let
$$
b_n = m_2(\ImB_2(n)).
$$
\ben
\item[(a)] We have
$$
\sum_{n\ge0} b_n\, q^n = \frac{q^3}{(1-q)(1-q^2)}\ \prod_{i\ge0} \frac{1}{{1-q^{2^i}}}.
$$
\item[(b)] For $n\ge0$,
$$
b_n = \sum_{i\ge1} \left\lfloor \frac{i-1}{2} \right\rfloor |\cB(n-i)|.
$$
\item[(c)] For $n\ge0$,
$$
b_n = \sum_{i=0}^{\lfloor n/2 \rfloor} \left\lfloor \frac{(n-1-2i)^2}{4} \right\rfloor |\cB(i)|.
$$
\item[(d)] For $n\ge0$,
$$
(n-3)\,b_n = \sum_{k=1}^n \left(\beta(k)+(-1)^k+2 \right)b_{n-k}.
$$
\item[(e)]  Keeping the notation of Theorem~\ref{m1B}, for $n\ge0$, 
$$
\De b_n = a_{\fl{n/2}+1}.
$$
\item[(f)]  For $n\ge0$,
$$
a_n = b_n + b_{n+1}.
$$
\een
\eth
\bprf
(a)  Using Lemma~\ref{2^k:ImB} with $k=1$ gives
\beq
\label{b=BR}
b_n=|\cBR_{1,2}(n)|.
\eeq
If the $1$ is rooted at position $j$ and the $2$ at position $k$ then
$$
\sum_{n\geq 0}b_n\,q^n
=\sum_{j\geq 1}q^j\ \sum_{k\ge1} q^{2k}\ \prod_{i\geq 0}\frac{1}{1-q^{2^i}}
= \frac{q}{1-q}\ \cdot  \frac{q^2}{1-q^2}\ \cdot \prod_{i\ge0} \frac{1}{{1-q^{2^i}}},
$$
which is what we wished to prove.

\medskip

(b)  One can give both algebraic and combinatorial proofs similar to those of Theorem~\ref{m1B} (b).  For the former we equate coefficients in (a).  For the latter, we use~\eqref{b=BR} and write $\la\in\cBR_{1,2}(n)$ as $\la =\mu\oplus\nu$ where $\nu$ consists of the $\hat{1}$ and all $1$'s to its right as well as the $\hat{2}$ and all $2$'s to its right.  The reader can easily supply the details.

\medskip

(c)  Again, the generating function proof parallels that of Theorem~\ref{a=BR} (c).
The combinatorial proof is also similar, writing  $\la\in\cBR_{1,2}(n)$ as $\la =\mu\oplus\nu$ where $|\mu|= 2i$  and $\nu$ contains  all $1$'s and the $\hat{1}$ as well as the $\hat{2}$ and all $2$'s to its right. 

\medskip


(d)  As in the demonstration of Theorem~\ref{a=BR} (d), one can obtain this result using logarithmic differentiation.


\medskip

(e)   We know the generating functions for $b_n$ and $a_n$ from parts (a) of this theorem and the previous one, respectively.  So it is  easy to compute the corresponding series for $b_{n+1}-b_n$ and $a_{\fl{n/2}+1}$ and check that they are the same.  

For the combinatorial proof, first note that
\beq
\label{Db=BR}
\Delta b_n = |\cBR_2(n)|,
\eeq
where the proof is similar to that of equation~\eqref{Da=BR}.  So, using~\eqref{a=BR}, it suffices to find a bijection $f:\cBR_2(n)\rightarrow\cBR_{1,1}(\fl{n/2}+1)$.  Take $\la\in\cBR_2(n)$ and write
$$
\la = \mu\oplus\nu,
$$
where $\nu$ contains all the $1$'s in $\la$ and suppose $|\nu|=i$.  Now let
$$
f(\la) = \mu/2 \oplus (\hat{1})\oplus (1^{\fl{i/2}})
$$
where the $\hat{2}$ in $\mu$ becomes a $\hat{1}$ in $\mu/2$.  The reader should now be able to fill in the details of the rest of the proof.

  \medskip

(f) This follows easily by comparing the generating functions for $a_n$ and $b_n$.  But we prefer a combinatorial proof. 
Using the usual translation into rooted partitions, it suffices to create a bijection
$$
f:\cBR_{1,1}(n) \rightarrow \cBR_{1,2}(n) \uplus \cBR_{1,2}(n+1).
$$ 
Suppose $\la\in\cBR_{1,1}(n)$ and write
$$
\la=\mu\oplus\nu\oplus\pi
$$
where $\nu$ is the first $\hat{1}$ together with all the $1$'s between it and the second, while $\pi$ is the second $\hat{1}$ and all the ones afterward.  Let $k=|\nu|$ and let
$\nu'$ to be $\ce{k/2}$ copies of $2$ with the first part rooted.  Note
that 
$$
|\nu'|=\begin{cases}
 k & \text{if $k$ is even}, \\
 k+1 & \text{if $k$ is odd}.
\end{cases}
$$
Now let
$$
f(\la) = \mu\oplus\nu'\oplus\pi.
$$
It is easy to check that $f$ is well defined and to construct its inverse.  The reader can fill in the details.
\eprf

It should now be clear to the reader that we can use the techniques we have developed to write down formulas for the number of parts equal to $2^k$ in 
$\ImB_2(n)$ for any $k$, although the expressions will be more and more complicated as $k$ increases. 
To give a sense of the complexity, we will 
now consider, $c_n$, the number of parts equal to $4$.   But we will content ourselves with the generating function
and a formula for $\De c_n$.
We will use the Kronecker delta which, for any statement $S$, is given by
$$
\de(S)
=\begin{cases}
    1 & \text{if $S$ is true},\\
    0 & \text{if $S$ is false}.
\end{cases}
$$
\bth \label{m4B}
For $n\ge0$, let
$$
c_n = m_4(\ImB_2(n)).
$$
\ben
\item[(a)]  We have
$$
\sum_{n\geq 0}c_n\,q^n=
\frac{q^4(1+q+2q^2)}{(1-q^2)(1-q^4)}\,\prod_{i\ge0} \frac{1}{{1-q^{2^i}}}.
$$
\item[(b)] For $n\ge0$,
$$
c_n=\sum_{i\geq 2} \left(i-1 \right)\, |\cB(n-2i)|
+\sum_{i\geq 4} \left\lceil \frac{i-4}{4} \right\rceil\, |\cB(n-i)|.
$$
\item[(c)] For $n\ge0$,
$$
\De c_n =|\cBR_4(n)|+\de(\text{$n$ odd})\, |\cBR_2(n)|.
$$
\een
\eth
\bprf 
(a) The $k=2$ case of Lemma~\ref{2^k:ImB} gives
$$
c_n = |\cBR_{1,4}(n)| + |\cBR_{2,2}(n)|,
$$
and then we proceed as in parts (a) of the previous two theorems.

\medskip

(b) We can rewrite the factor in front of the product in (a) as
\begin{align*}
    \frac{q^4(1+q+2q^2)}{(1-q^2)(1-q^4)}
    &= \frac{q^4(1+q^2)}{(1-q^2)(1-q^4)} + \frac{q^4(q+q^2)}{(1-q^2)(1-q^4)}\\
    &= \frac{q^4}{(1-q^2)^2} + \frac{q^5}{(1-q)(1-q^4)}\\
    &= \sum_{n\geq 2} (n-1)\,q^{2n} + \sum_{n\geq 4} \left\lceil \frac{i-4}{4} \right\rceil\, q^n.
\end{align*}
Multiplying by $\prod_{i\ge0} 1/(1-q^{2^i})$, using part (a), and extracting the coefficient of $q^n$ finishes the proof.

(c)  The parts equal to $4$  in $\ImB_2(n)$ come from products in $\cB_2(n)$ of either a $1$ with a $4$, or a $2$ with a $2$.  The number of new products of the former type  counted by $\De c_n$ is $|\cBR_4(n)|$, analogous to equations~\eqref{Da=BR} and~\eqref{Db=BR}.  If $n$ is even then the map 
$f:\cB(n)\ra\cB(n+1)$ defined by
$f(\la)=\la\oplus(1)$ is a bijection and so there are no new products of $2$ and $2$.  Thus the proof is complete in this case.

If $n$ is odd then, since every $\la\in\cB(n)$ contains a $1$, we can define a map $g:\cB(n)\ra\cB(n+1)$ by writing $\la=\mu\oplus(1)$ and letting $g(\la)=\mu\oplus(2)$.
This will be a bijection onto the partitions in $\cB(n+1)$ with at least one $2$. These are the only $2$'s which can contribute a $2\cdot 2$ to $\De c_n$.  Furthermore, the number of such products which are created is just the number of $2$'s in $\la$.  This accounts for the $|\cBR_2(n)|$ term in the sum and we are finished.
\eprf

\section{$d$-ary partitions}
\label{dp}

Consider an integer $d\ge2$.  Say that $\la$ is a {\em $d$-ary partition} if all its parts are powers of $d$. Define
$\cP(n,d)$ to be the set of $d$-ary partitions of size $n$ and
$$
\cP_k(n,d) = \{\la\in\cP(n,d) \mid \ell(\la)\ge k\}.
$$
Also, set
$$
\ImP_k(n,d) = \pre_k (\cP_k(n,d)).
$$
Obviously when $d=2$ we recover the binary partitions.
In this section we will study  $\ImP_2(n,d)$.  

First we show that $\pre_2$ is injective on $\cP_k(n,d)$.  We begin with a lemma which is a generalization of Lemma~\ref{2^k:ImB} to all partitions.
\ble
Suppose that $\pre_2(\la)=\mu$.  Then for any $k$ we have
\begin{equation}
\label{m:rec}
 m_k(\mu) = \sum_{\substack{e f= k\\ e<f}} m_e(\la) m_{f}(\la) + \de(\text{$k$ is a square})\binom{m_{\sqrt{k}}(\la)}{2}.
\end{equation}
\ele
\bprf
We get a copy of $k$ in $\mu$ for each pair $e,f$ of parts in $\la$ such that $e f=k$.
To get each pair exactly once we assume $e\le f$.  If $e<f$ then this gives one of the terms of the sum.  If $e=f$ then $k=e^2$ and the number of ways to pick two $e$'s from $\la$ is the binomial coefficient.
\eprf

\bth
For any $n, d\ge2$, the map $\pre_2$ is injective on $\cP_k(n,d)$.
\eth
\bprf
Suppose that $\pre_2(\la)=\mu$.  It suffices to show that the multiplicities in $\mu$ determine those in $\la$ uniquely.  For ease of notation, let
Let 
$$
m = m_1(\mu)=\binom{m_1(\la)}{2}.
$$
We have two cases depending on whether $m$ is nonzero.

If $m>0$ then the previous displayed equation 
uniquely determines $m_1(\la)$ as a nonzero integer.
To determine $m_k(\la)$ note that, by induction on $k$, we can assume that every $m_j(\la)$ in equation~\eqref{m:rec} with $j<k$ has been determined.  Furthermore, $m_k(\la)$ only appears in the term $m_1(\la)m_k(\la)$ 
of this equation  and its coefficient is nonzero.  Thus equation~\eqref{m:rec} is linear in $m_k(\la)$ and so has a unique solution as desired.

If $m=0$ then $m_1(\la)=0$ or $1$.  Since $\la$ is $d$-ary we have that $n=|\la|$ satisfies $n\equiv m_1(\la)\ (\Mod d)$.  So we can tell the value
of $m_1(\la)$ from $n$.  If $m_1(\la)=1$, then the same argument as in the previous paragraph shows that $\la$ is determined.
If $m_1(\la)=0$ then consider the partition $\la/d$ obtained by dividing every part of $\la$ by $d$.  So $\la/d$ is a $d$-ary partition of $n/d$.  And, by induction on $n$, we can assume that $\la/d$ can be reconstructed from $\pre_2(\la/d) = \mu/d^2$.  Finally, multiplying every part of $\la/d$ by $d$ gives $\la$.
\eprf

Note that the same proof just given can be used to show that $\pre_2$ is injective on the set of partitions of $n$ with at least two $1$'s.
We can now generalize some of the results of the previous section on binary partitions to $d$-ary ones.  Since the proofs follow the same lines as when $d=2$, they will be omitted.  We will need a generalization of the $\be$ function defined above, namely
$$
\be(n,d) = \sum_{d^i | n} d^i.
$$
\bth \label{dary}
For $n\ge0$, let
\begin{align*}
  a_n(d) & =m_1(\ImP_2(n,d)),\\ 
  b_n(d) & =m_d(\ImP_2(n,d)).
\end{align*} 
\ben
\item[(a)] We have
\begin{align*}
   \sum_{n\ge0} a_n(d)\, q^n 
   &=  \frac{q^2}{(1-q)^2}\ \prod_{i\ge0} \frac{1}{{1-q^{d^i}}},\\
   \sum_{n\ge0} b_n(d)\, q^n 
   &=  \frac{q^{d+1}}{(1-q)(1-q^d)}\ \prod_{i\ge0} \frac{1}{{1-q^{d^i}}}.  
\end{align*}
\item[(b)]  For $n\ge0$,
\begin{align*}
   a_n(d) &= \sum_{i\ge 1} (i-1)\ |\cP(n-i,d)|,\\
   b_n(d) &= \sum_{i\ge 1} \flf{i-1}{d}\ |\cP(n-i,d)|.
\end{align*}

\item[(c)]  For $n\ge0$,
\begin{align*}
 a_n(d) &= \sum_{i=1}^{\lfloor n/d \rfloor} \binom{n-d\,i}{2}\ |\cP(i,d)|,\\  
 b_n(d) &= \sum_{i\ge1} \left(i\flf{i-1}{d} - d \binom{\lfloor (i-1)/d \rfloor+1}{2} \right)\ |\cP(n-i,d)|,\\  
\end{align*}

\item[(d)] For $n\ge0$,
\begin{align*}
   (n-2)\,a_n(d)&=\sum_{k=1}^n ( \be(k,d)+2)\,a_{n-k},\\ 
(n-d-1)\,b_n(d)&=\sum_{k=1}^n [\be(k,d)+1+\de(d|(n-1))\cdot d]\,a_{n-k}.
\end{align*}

\item[(e)]  For $n\ge0$,
\begin{align*}
 \De a_n(d) &= |\cP(dn-d,d)|,\\ 
 \rule{150pt}{0pt}
 \De b_n(d) &= a_{\fl{n/d}} + 1.
  \rule{150pt}{0pt}\qed
\end{align*}
\een
\eth

\section{Color partitions}
\label{cp}

Partitions with $n$ copies of $n$, also referred to in the literature as $n$-color partitions, are partitions in which the part of size $n$ appears in $n$ different colors: $n_1, n_2, \ldots, n_n$. The parts of an $n$ color partition are ordered lexicographically.
Implicitly, $n$-color partitions appear in work of MacMahon \cite[Chapters 11-12]{mm}.
 As noted by Andrews and Paule~\cite{ap:ncn}, $n$-color partitions  were also used implicitly in Regime III of the hard
hexagon model, see the paper of Andrews, Baxter, and Forrester~\cite{hhm}.  They were first studied explicitly by Agarwal and Andrews~\cite{aa:ncn} and have since been considered by a number of authors.

Here, we consider color binary partitions in which a part of size $2^n$ can come in $n+3$ different
colors denoted by subscripts: $2^n_1$, $2^n_2$, $\ldots$, $2^n_{n+3}$. Totally order the parts by
$$
1_1<1_2<1_3<2_1<2_2<2_3<2_4<4_1<4_2<4_3<4_4<4_5<\cdots\, .
$$
We call such partitions {\em $(n+3)$-color binary partitions},
and let
$\cQ(m)$ be the set of such partitions of $m$ into distinct parts.
 For example, 
   \begin{align*}
    \cQ(4) =  \{ &(4_5),\ (4_4),\ (4_3),\ (4_2),\ (4_1),\ (2_4,2_3),\ (2_4,2_2),\ (2_4,2_1),\ (2_4,1_3,1_2),\ (2_4,1_3,1_1),\\
       & (2_4,1_2,1_1),\ (2_3,2_2),\  (2_3,2_1),\ (2_3,1_3,1_2),\ (2_3,1_3,1_1),\ (2_3,1_2,1_1),\ (2_2,2_1),\\
       & (2_2,1_3,1_2),\ (2_2,1_3,1_1),\ (2_2,1_2,1_1),\ (2_1,1_3,1_2),\ (2_1,1_3,1_1),\ (2_1,1_2,1_1)\}.
   \end{align*}

We can now give  new identities for the sequence $a_n$ appearing in  Theorem~\ref{m1B}. 
We also obtain an expression for the $\be(n)$ function defined by~\eqref{be(n)}.
\begin{theorem}
\label{th8}
For $n\ge0$, let
$$
Q_n =|\cQ(n)|.
$$
\ben
\item[(a)]  We have
$$
\sum_{n\geq 0} a_n\,q^n  = q^2\,\prod_{i\geq 0} (1+q^{2^i})^{i+3}.
$$
\item[(b)] For $n\geq 0$, 
$$
a_{n+2}=Q_n.
$$
\item[(c)] For $n\ge0$,
$$
a_{n+2}= \sum_{\la\in\cB(n)}\, \prod_{i\ge0} \binom{i+3}{m_{2^i}(\la)}.
$$
\een
\end{theorem}
\begin{proof}
(a)
Since every integer has a unique base $2$ expansion, we have
$$
\frac{1}{1-q} = \prod_{i\geq 0} (1+q^{2^i}).
$$
Replacing $q$ with $q^{2^n}$ in the previous equation gives
$$  
\frac{1}{1-q^{2^n}} = \prod_{i\geq 0} (1+q^{2^{i+n}}).
$$ 
Using Theorem~\ref{m1B} (a) and the two previous equations   gives
   \begin{align*}
       \sum_{n\geq 0} a_n\,q^n & = q^2\, \frac{1}{(1-q)^2} \prod_{n\geq 0} \frac{1}{1-q^{2^n}} \\
       & = q^2\, \prod_{i\geq 0} (1+q^{2^i})^2\, \prod_{n\geq 0} \prod_{i\geq 0} (1+q^{2^{i+n}}) \\
        & = q^2\, \prod_{i\geq 0} (1+q^{2^i})^2\, \prod_{i\geq 0} (1+q^{2^{i}})^{i+1} \\
        & = q^2\,\prod_{i\geq 0} (1+q^{2^i})^{i+3}.
   \end{align*}

\medskip

(b)  We give both an algebraic and a combinatorial proof.  For the former,
 elementary techniques in the theory of partitions as in Andrews' text~\cite{and:ttp} give the following generating function
\beq
\label{Q_n:prod}
\sum_{n\geq 0} Q_n\,q^n = \prod_{i\geq 0} (1+q^{2^i})^{i+3}.
\eeq
Substituting this for the product in part (a) and equating coefficients of $q^{n+2}$ finishes the algebraic demonstration.

We now give a proof, using~\eqref{a=BR}, by constructing a bijection $f:\cBR_{1,1}(n+2)\rightarrow\cQ(n)$.  See Example~\ref{BQ:ex} for an illustration of the map.  Write $\la\in\cBR_{1,1}(n+2)$ as
\beq
\label{lmnp}
\la = \mu \oplus (\hat{1}) \oplus \nu \oplus (\hat{1}) \oplus \pi,
\eeq
where $\nu$ contains all the $1$'s between the first and second $\hat{1}$, and $\pi$ contains all the $1$'s after the second.  We form a color partition $\mu'$ from $\mu$ as follows.  For each $t\ge0$ suppose $2^t$ occurs in $\mu$ with multiplicity $s$ ($s$ depending on $t$).  Write $s\cdot2^t$ in binary notation and color each power of $2$ with $t+3$, adding these elements to $\mu'$  for each value of $t$.  From $\nu$ we create $\nu'$ by writing $|\nu|$ in binary and subscripting the summands with $2$.  Similarly form $\pi'$ from $\pi$ only using subscripts $1$.  Finally, let
\beq
\label{flmnp}
f(\la) = \mu'\oplus \nu'\oplus \pi'.
\eeq
It is easy to check that this map is a well-defined bijection.
We note the similarity of this bijection  to the usual one  used to prove that the number of partitions of $n$ into odd parts equals the number into distinct parts.

\medskip

(c)  Example~\ref{a:m} illustrates this formula.
    By employing the binomial expansion, the generating function~\eqref{Q_n:prod} can be rewritten as
$$
        \sum_{n\geq 0} Q_n\,q^n = \prod_{n\geq 0} \sum_{j\ge0} \binom{n+3}{j} q^{j\cdot2^n}.
$$
Collecting all the terms in the product contributing to a given power $q^n$ gives
$$
    \sum_{n\geq 0} Q_n\,q^n  
    = \sum_{n\geq 0} q^n \sum_{\la\in\cB(n)} \prod_{i\ge0} \binom{i+3}{m_{2^i}(\la)}.
$$
Combining this with part (b) finishes the proof.
\end{proof}

\begin{example}
\label{BQ:ex}
To illustrate the bijection for the proof of part (b) of the previous theorem, suppose
$$
\la= (2,2,2,1,1,\hat{1},1,\hat{1},1,1) = (2,2,2,1,1)\oplus(\hat{1})\oplus(1)\oplus(\hat{1})\oplus(1,1)
$$
so that
\begin{align*}
    \mu&=(2,2,2,1,1),\\
    \nu&=(1),\\
    \pi&=(1,1).
\end{align*}
In $\mu$ we have three copies of $2$ and two copies of $1$, so we write
\begin{align*}
3\cdot 2^1 &=4 + 2.\\
2\cdot 2^0 &= 2.
\end{align*}
Thus
$$
\mu' = (4_4, 2_4, 2_3).
$$
Similarly,
\begin{align*}
    |\nu|&= 1,\\
    |\pi|&=2,
\end{align*}
so that
\begin{align*}
    \nu'&= (1_2),\\
    \pi'&=(2_1).
\end{align*}
Finally
$$
f(\la) = (4_4, 2_4, 2_3) \oplus (1_2)\oplus (2_1) = (4_4, 2_4, 2_3, 2_1, 1_2).
$$
\end{example}

\begin{example}
\label{a:m}
The summation on the right-hand side of part (c) of the previous result encompasses all binary partitions of $n$, but not every term contributes because some are zero. 
For $m_{2^k}(\la)>k+3$, the binomial coefficient becomes zero,
allowing us to focus solely on binary partitions of $n$ in which each part $2^k$ has multiplicity at most $k+3$.
For example, the binary partitions of $4$ in which each part $2^k$ has the multiplicity at most $k+3$ are:
$$(4),\ (2,2),\ (2,1,1).$$
Thus
$$
    a_6 = \binom{2+3}{1} + \binom{1+3}{2} + \binom{1+3}{1} \binom{0+3}{2} 
         = 5 + 6 + 4\cdot 3 
         = 23.
$$
\end{example}

We can generalize the construction just given to give expressions for $2^d$-ary partitions.
Denote by $\cQ(m,2^d)$ the set of color binary partitions of $m$ into distinct parts in which a part of size $2^n$ can come in $\lfloor n/d \rfloor+3$ different colors.  As before, parts are ordered lexicographically.  For example, when $d=2$ then we are allowed parts
$$
1_1<1_2<1_3<2_1<2_2<2_3<4_1<4_2<4_3<4_4<8_1<8_2<8_3<8_4<\cdots\, .
$$
By way of example,
   \begin{align*}
    \cQ(4,2^2) =  \{ &(4_4),\ (4_3),\ (4_2),\ (4_1),\ (2_3,2_2),\  (2_3,2_1),\ (2_3,1_3,1_2),\ (2_3,1_3,1_1),\ (2_3,1_2,1_1),\\ 
       & (2_2,2_1),\ (2_2,1_3,1_2),\ (2_2,1_3,1_1),\ (2_2,1_2,1_1),\ (2_1,1_3,1_2),\ (2_1,1_3,1_1),\ (2_1,1_2,1_1)\}.
   \end{align*}
We have the following generalization of the previous theorem. Its proof is similar and so omitted.
\begin{theorem}
    Let $d$ be a positive integer. For $n\geq 0$, let
    $$
    Q_n(2^d) = |\cQ(n,2^d)|.
    $$
    \begin{enumerate}
    \item[(a)] We have
$$
\sum_{n\geq 0} a_n(2^d)\,q^n  = q^2\,\prod_{i\geq 0} (1+q^{2^i})^{\fl{i/d}+3}.
$$
    \item [(b)] For $n\ge0$,
    $$
     a_{n+2}(2^d)=Q_{n}(2^d).
    $$
    \item [(c)] For $n\ge0$,
    $$
    \hspace{120pt}
    a_{n+2}(2^d)= \sum_{\la\in\cB(n)} \prod_{i\ge0} \binom{\lfloor i/d \rfloor+3}{m_{2^i}(\la)}.
    \hspace{120pt}\qed
    $$
    \end{enumerate}
\end{theorem}

We can use color partitions to study $d$-ary partitions where $d$ is odd.
Consider color partitions into parts of the form $2^i\cdot(2d+1)^j$, in which a part of size $2^i$ can come in three different
colors and all other parts only have one color. 
As usual, parts are ordered lexicographically.
For example, when  $d=1$ we have parts
$$1_1<1_2<1_3<2_1<2_2<2_3<3_1<4_1<4_2<4_3<6_1<8_1<8_2<8_3<9_1<\cdots$$
We denote by $\cQ(m,2d+1)$ the set  of such partitions of $m$ into distinct parts.   To illustrate,
   \begin{align*}
       \cQ(4,3)=\{&(4_3),\ (4_2),\ (4_1),\ (3_1,1_3),\ (3_1,1_2),\ (3_1,1_1),\ (2_3,2_2),\ (2_3,2_1),\ (2_2,2_1)\\
       & (2_3,1_3,1_2),\ (2_3,1_3,1_1),\  (2_3,1_2,1_1),\ (2_2,1_3,1_2),\ (2_2,1_3,1_1),\  (2_2,1_2,1_1),\\
       & (2_1,1_3,1_2),\ (2_1,1_3,1_1),\  (2_1,1_2,1_1)\}.
   \end{align*}

\begin{theorem}
    Let $d\ge 1$ be an integer. For $n\geq 0$, let
    $$
    Q_n(2d+1) = |\cQ(n,2d+1)|.
    $$
    \begin{enumerate}
    \item[(a)] We have
$$
\sum_{n\geq 0} a_n(2d+1)\,q^n  = q^2\, \prod_{n\geq 0} (1+q^{2^n})^2 \, \prod_{i,j\geq 0} (1+q^{2^i\,(2d+1)^j}).
$$
    \item [(b)] For $n\ge0$,
    $$
     a_{n+2}(2d+1)=Q_{n}(2d+1).
    $$
    \end{enumerate}
\end{theorem}

\begin{proof}
(a) From part (b), for which we will give an independent proof, it suffices to show that
   \begin{align*}
       \sum_{n\geq 0} Q_{n}(2d+1)\,q^n = \prod_{n\geq 0} (1+q^{2^n})^2 \, \prod_{i,j\geq 0} (1+q^{2^i\,(2d+1)^j}).
   \end{align*}
But this follows by considering which parts can appear in a color partition counted by $Q_n(2d+1)$.

 \medskip
 
 (b)
 We give a combinatorial argument similar to that for Theorem \ref{th8} (b). 
Let $\cPR_{1,1}(n,d)$ be the set of $d$-ary partitions of $n$ with two of its $1$'s rooted.
 Then we wish to construct a bijection $f:\cPR_{1,1}(n+2,2d+1)\rightarrow\cQ(n,2d+1)$.
 We decompose $\la\in\cPR_{1,1}(n+2,d)$ as a direct sum exactly as in~\eqref{lmnp}.
Suppose $(2d+1)^t$ occurs in the summand $\mu$ with multiplicity $s$.  Then for each power $2^r$ in the binary expansion of $s$, we put a part $2^r\cdot (2d+1)^t$ with color $1$ into a partition $\mu'$.
The summands $\nu$ and $\pi$ are replaced by binary partitions $\nu'$ and $\pi'$ with subscripts $2$ and $3$, respectively, in the same way as done in the proof of Theorem \ref{th8} (b).  Finally, we define $f(\la)$ by~\eqref{flmnp}.  As usual, the details of showing that this is a well-defined map and a bijection are left to the reader.
\eprf

\section{Conjectures and future work}

We end with some conjectures and a direction for future research in the hopes that the reader will be interested in pursuing them.  In a number of cases, the conjectures will follow easily if Conjecture~\ref{prek:con} can be proved.

In the previous two sections we dealt with $d$-ary partitions.  By contrast, a partition is {\em $d$-regular} if it contains no part which is a multiple of $d$.  For example, the $2$-regular partitions are the ones into odd parts.
Let
$$
r_{d,k}(n) =|\{\la \mid \text{$\la\in\ImP_k(n)$  is $d$-regular}\}|.
$$

\begin{conjecture}
The value of $r_{d,2}(n)-r_{d,3}(n)$ for $d=2$, $3$, $4$, and $5$  are shown in the columns of the following table.  These values depend on the congruence class of $n$ modulo $2$, $6$, $4$, and $10$, respectively.  The first column of the table gives the congruence class for $n$. 
$$
\begin{array}{r||l|l|l|l}
    &r_{2,2}(n)-r_{2,3}(n)      &  r_{3,2}(n) - r_{3,3}(n) &r_{4,2}(n) - r_{4,3}(n)  &r_{5,2}(n) - r_{5,3}(n)   \\
n \pmod m & m = 2               & m = 6                    & m = 4                    & m = 10\\
    \hline \hline
0   &\fl{(n+2)/4}               &2\fl{n/6}                  & \fl{n/4}                 &3\fl{n/10}\\
1   &0                          &\fl{n/6}                   & \fl{n/4}                 &4\fl{n/10}\\
2   &                           &\fl{n/6}+1                 & \fl{n/4} + 1             &3\fl{n/10}\\
3   &                           &2\fl{n/6} +1               & \fl{n/4} + 1             &3\fl{n/10}+1\\
4   &                           &\fl{n/6}+1                 &                          &3\fl{n/10}+1\\
5   &                           &\fl{n/6}+1                 &                          &3\fl{n/10}+2\\
6   &                           &                           &                          &4\fl{n/10}+2\\             
7   &                           &                           &                         &3\fl{n/10}+2\\  
8   &                           &                           &                         &3\fl{n/10}+2\\
9   &                           &                           &                         &3\fl{n/10}+3
\end{array}
$$
\end{conjecture}

\begin{conjecture}
If $p$ is prime then
$$
r_{p,1}(n) - r_{p,2}(n) =\begin{cases}
    0 & \text{if $p\mid n$},\\
    1 & \text{otherwise}.
\end{cases} 
$$
\end{conjecture}

The next conjecture is interesting because one side of the hoped-for equality depends on a parameter $k$ while the other does not.  If, $m<n$ then it will be useful to use the notation
\begin{align*}
\cP[m,n]&=\biguplus_{i=m}^n \cP(i),\\
\ImP_k[m,n]&=\biguplus_{i=m}^n \ImP_k(i).
\end{align*}
\begin{conjecture}
For $n\ge0$ and $k\ge2$,
\begin{align*}
  \De^{k-1}   m_4(\ImP_k[n,n+1])&=m_2(\cP(n))+m_4(\cP[n,n+1]),\\
  \De^{k-1}   m_6(\ImP_k(n)) &= m_6(\cP(n)) + m_2(\cP(n-2))-m_3(\cP(n-2)),\\
  \De^{k-1}   m_6(\ImP_k[n,n+1])&=m_3(\cP(n))+m_6(\cP[n,n+1]),\\
  \De^{k-1}   m_9(\ImP_k[n,n+2]) &= m_3(\cP(n)) + m_9(\cP[n,n+2]),\\
  \De^{k-1}   m_{10}(\ImP_k[n,n+1])&=m_5(\cP(n))+m_{10}(\cP[n,n+1]),\\
  \De^{k-1}   m_p(\ImP_k(n)) &= m_p(P(n)),
\end{align*}
where the last equality holds for $p=1$ or $p$ prime.
\end{conjecture}

We now make connections with some sequences in the OEIS.  To do so, we let  $\chi(\cS)$ be the total number of different parts in the set of partitions $S$, so multiple appearances of the same part are only counted once.  On the other hand, we let $\tau(\cS)$ be the total number of parts in $S$ counted with multiplicity.
\begin{conjecture}
For $n\ge0$,
\begin{align*}
\chi(\ImP_2(n)) &= A227800(n+1),\\
\chi(\ImB_2(n)) &= A126236(n),\\
\chi(\ImP_3(n)) &=1+ A213213(n),\\
\tau(\ImP_2(n)) &= A258472(n).
\end{align*}
\end{conjecture}

Let us end by discussing how our ideas could be extended to other symmetric polynomials.
There are many standard bases for the algebra of symmetric polynomials, for example, the monomial, complete homogeneous, and Schur bases.  
For more information about symmetric polynomials, see the texts of Macdonald~\cite{mac:sfh},
Sagan~\cite{sag:sg,sag:aoc}, or Stanley~\cite{sta:ec2}.
One could evaluate any of these on a partition and use the summands as parts. 
To illustrate, consider the complete homogeneous symmetric polynomial
$$
h_k(x_1,x_2,\ldots,x_\ell) = \text{ the sum of all degree $k$ monomials in the $x_i$.}
$$
As an example,
$$
h_3(x_1,x_2) = x_1^3 + x_1^2 x_2 + x_1 x_2^2 + x_2^3.
$$
So for the partition $\la=(4,3)$ we have
$$
h_3(4,3) = 4^3 + 4^2\cdot 3 + 4 \cdot 3^2 + 3^3
$$
and we could define a partition $\prh(3,1)$ by
$$
\prh(3,1) = (64,48,36,27).
$$
It may well be that partitions formed from other bases will also have interesting properties.

\section{Declarations}

{\bf Funding and/or Conflicts of interests/Competing interests.}   The authors did not receive support from any organization for the submitted work.  The authors have no relevant financial or non-financial interests to disclose.  On behalf of all authors, the corresponding author states that there is no conflict of interest.

\nocite{*}
\bibliographystyle{alpha}

\end{document}